\documentclass[12pt]{amsart}
\usepackage{amsmath, amssymb, amsfonts, amscd, amsbsy, amsthm}  
\usepackage{latexsym, color, graphicx}
\usepackage{bm}
\usepackage[normalem]{ulem}

\numberwithin{equation}{section}

\theoremstyle{plain}
\newtheorem{Theorem}{Theorem}[section]
\newtheorem{Proposition}[Theorem]{Proposition}

\newtheorem{Problem}[Theorem]{Problem}

\theoremstyle{definition}
\newtheorem{Definition}[Theorem]{Definition}

\theoremstyle{remark}
\newtheorem{Remark}{{\bf Remark}}
\newtheorem{Example}{Example}

\newcommand{\R}{{\mathbb R}}

\newcommand{\dps}{\displaystyle}

\begin{document}

\title{Maximization of the first Laplace eigenvalue of a finite graph}

\author[Gomyou]{Takumi Gomyou}
\thanks{Department of Mathematics, Osaka University,
Osaka 560-0043, Japan, gomyou@cr.math.sci.osaka-u.ac.jp}

\author[Nayatani]{Shin Nayatani}
\thanks{Graduate School of Mathematics, Nagoya University,
Chikusa-ku, Nagoya 464-8602, Japan, nayatani@math.nagoya-u.ac.jp}

%\thanks{* Corresponding author}
%\address{Graduate School of Mathematics, Nagoya University, Chikusa-ku, Nagoya 464-8602, Japan}
%\email{nayatani@math.nagoya-u.ac.jp}
%\thanks{Partly supported by the Grant-in-Aid
%for Scientific Research, The Ministry of Education,
%Science, Sports and Culture, Japan.}
%\subjclass[2010]{Primary~20F65; Secondary~58E20, 20P05.}
\keywords{graph, Laplacian, eigenvalue, embedding}

%\date{}

\maketitle

\begin{abstract}
Given a length function on the edge set of a finite graph, we define a vertex-weight and an edge-weight in terms of it 
and consider the corresponding graph Laplacian. 
In this paper, we consider the problem of maximizing the first nonzero eigenvalue of this Laplacian over all edge-length 
functions subject to a certain normalization. 
For an extremal solution of this problem, we prove that there exists a map from the vertex set to a Euclidean space 
consisting of first eigenfunctions of the corresponding Laplacian so that the length function can be explicitly expressed 
in terms of the map and the Euclidean distance.  
This is a graph-analogue of Nadirashvili's result related to first-eigenvalue maximization problem on a smooth surface.
We discuss simple examples and also prove a similar result for a maximizing solution of the G\"oring-Helmberg-Wappler 
problem. 
\end{abstract}

\section*{Introduction}

%In this paper, we investigate embeddings of finite graphs by first eigenfunctions of the Laplacian when the parameters defining the Laplacian are chosen so as to maximize the first nonzero eigenvalue. 

%We consider two types of first-eigenvalue maximization problems. 

%In the second problem, newly introduced in this paper, 

Let $G=(V,E)$ be a finite graph, where $V$ (resp. $E$) is the set of vertices (resp. edges). 
For any pair of a vertex-weight $m_0\colon V\to \R_{>0}$ and an edge-weight $m_1\colon E\to \R_{\geq 0}$, one can define the corresponding Laplacian $\Delta_{(m_0,m_1)}$. 
If an edge-length function $l\colon E\to \R_{>0}$ is given, following Fujiwara \cite{Fujiwara}, we define weights $m_0$ and $m_1$ suitably in terms of $l$ and thus the Laplacian $\Delta_l := \Delta_{(m_0,m_1)}$. 
Then we regard the length function $l$ as a variable and consider the problem of maximizing the first nonzero eigenvalue of the Laplacian $\Delta_l$ over all choices of $l$ subject to a certain normalization. 
A maximizing length function does not always exist, but if one exists, 
%non-degenerate everywhere, 
we prove that there exists a map $\varphi\colon V \to \R^N$ consisting of first eigenfunctions of the corresponding Laplacian so that the length function can be explicitly expressed in terms of the map $\varphi$ and the Euclidean distance. 
Here, $N$ is some positive integer less than or equal to the multiplicity of the first nonzero eigenvalue.
More generally, we prove the same assertion for any extremal solution.

A similar first-eigenvalue maximization problem was formulated by G\"{o}ring-Helmberg-Wappler \cite{GoringHelmbergWappler1, GoringHelmbergWappler2}. They fix a vertex-weight $m_0$ and an edge-length function $l$, define the weighted Laplacian $\Delta_{(m_0,m_1)}$ in terms of $m_0$ and a choice of edge-weight $m_1$, and maximize the first nonzero eigenvalue of $\Delta_{(m_0,m_1)}$ over all choices of $m_1$ subject to a normalization. 
They verified that a maximizing edge-weight always exists though it may vanish on some edges. 
If this degeneracy does not happen, then there exists a map similar to the above one but with the better property that the length function exactly coincides with the pull-back of the Euclidean distance by the map. 
Historically, Fiedler \cite{Fiedler2} considered a special case of this problem, fixing $m_0\equiv 1$ and maximizing the first nonzero eigenvalue of the weighted Laplacian over all choices of $m_1$ with mean value one. He calls the maximum value as the {\em absolute algebraic connectivity} of the graph. The condition on $m_1$ is identical to G\"{o}ring et al.'s with the choice $l\equiv 1$.

Regarding the largeness of the first positive eigenvalue of the graph Laplacian, another direction of study involves expander graphs, that is, a sequence of regular graphs with fixed degree and growing number of vertices, ensuring that their first positive eigenvalues are uniformly bounded away from zero. Note that for a generic sequence of such graphs, the first positive eigenvalues will decay to zero. For this subject, the reader is referred to the comprehensive exposition \cite{HooryLinialWigderson}. 

In the background of the present work, there also is
%we recall 
%%\emii{the problem of first Laplace eigenvalue maximization} 
%(
%or, 
the problem of maximizing the first Laplace eigenvalue
%) 
on a manifold. The origin of the problem is the celebrated work of Hersch \cite{Hersch}, who proved that on the two-sphere, the scale-invariant quantity $\lambda_1(g) \mathrm{Area}(g)$, where $g$ is a Riemannian metric and $\mathrm{Area}(g)$ is the area of $g$, is maximized by the metrics of round spheres in $\R^3$. Motivated by Hersch's result, Berger \cite{Berger} asked whether on an arbitrary compact manifold of dimension $n$, the scale-invariant quantity $\lambda_1(g) \mathrm{Vol}(g)^{2/n}$ was bounded from above by a constant depending only on the manifold. 
Then Urakawa \cite{Urakawa} found an explicit family of metrics $g_t$, $t>0$, on the three-sphere for which the quantity 
$\lambda_1(g_t) \mathrm{Vol}(g_t)^{2/3}$, where $\mathrm{Vol}(g)$ is the volume of $g$, diverges to infinity as $t\to\infty$. This answers in the negative the above question of Berger. 
Later, Colbois-Dodziuk \cite{ColboisDodziuk} proved, by gluing Urakawa's metrics or their analogues on higher dimensional spheres to a given manifold, that the quantity $\lambda_1(g) \mathrm{Vol}(g)^{2/n}$ is unbounded on any compact manifold. 
On the other hand, for surfaces, rich progress has been made on the Berger problem and also the problem of maximizing $\lambda_1(g) \mathrm{Area}(g)$ (which is equivalent to maximizing $\lambda_1(g)$ under the normalization $\mathrm{Area}(g)=1$). The first result (after Hersch) is due to Yang-Yau \cite{YangYau}, who proved that $\lambda_1(g) \mathrm{Area}(g)\leq 8\pi(\gamma+1)$, where $\gamma$ is the genus of the surface. (Later El Soufi-Ilias \cite{ElSoufiIlias} remarked that the constant on the right-hand side could be improved to $8\pi\left[\frac{\gamma+3}{2}\right]$.) Since the work of Yang-Yau, many important works were done and many interesting results were proved. For these, we refer the reader to \cite{LiYau, MontielRos, Nadirashvili, JLNNP, Petrides, NayataniShoda, MatthiesenSiffert, Ros, Karpukhin}. Here, we only mention the result of Nadirashvili, which states that 
if a Riemannian metric $g$ is an extremal solution of the problem of maximizing $\lambda_1(g) \mathrm{Area}(g)$ on a compact surface $M$, then there exist first eigenfunctions $\varphi_1,\dots, \varphi_N$ of the corresponding Laplacian such that $\varphi = (\varphi_1,\dots, \varphi_N)\colon M\to \R^N$
is an isometric immersion; Therefore, $\varphi$ is a minimal immersion into the sphere $S^{N-1}(\sqrt{2/\lambda_1(g)})$ by the Takahashi theorem. 

Our main result is regarded as a discrete analogue of this result of Nadirashvili. 

For the study of the eigenvalues of the graph Laplacian from the perspective of differential geometry, the reader is referred to the comprehensive monograph \cite{Chung}. 

This paper is organized as follows.
In Section 1, we consider the graph Laplacian defined from an edge-length function and formulate a maximization problem for the first nonzero eigenvalue of this Laplacian. We state a Nadirashvili-type theorem %\emii{(Theorem \ref{thm-embed-var-1})} 
for an extremal solution of this problem. 
This is the main theorem and it is proved in Section 2, where a weak converse of the theorem is also proved. 
In Section 3, we find extremal solutions of the above problem for some graphs by numerical means. We also find the maps by first eigenfunctions of the main theorem. 
In Section 4, we prove an analogous result for a maximizing solution of the G\"oring-Helmberg-Wappler problem. 

\section{Preliminaries, Problem and Main Result} 
%G\"{o}ring-Helmberg-Wappler's problem}
Let $G=(V, E)$ be a finite connected graph, where $V$ and $E$ are the sets of vertices and (undirected) edges, respectively. We assume that $G$ is simple, that is, that $G$ has no loops nor multiple edges. Let $uv$ denote the edge whose endpoints are $u$ and $v$. $G$ being undirected means that $uv=vu$.
Choose weights $m_0\colon V \to \R_{>0}$ and $m_1\colon E \to \R_{\geq 0}$, and let $\R^V$ denote the set of functions $\varphi\colon V\to \R$, equipped with the inner product 
$$
\langle \varphi_1, \varphi_2 \rangle = \sum_{u\in V} m_0(u)\, \varphi_1(u)\, \varphi_2(u),
\quad \varphi_1, \varphi_2\in \R^V.
$$
Then the graph Laplacian is a positive symmetric linear operator 
$\Delta_{(m_0,m_1)} : \R^V \to \R^V$, defined by
\begin{equation}\label{weighted-Laplacian}
(\Delta_{(m_0,m_1)} \varphi) (u) = \sum_{v\sim u} \frac{m_1(uv) }{m_0(u)} \left( \varphi(u) - \varphi(v) \right),\quad u\in V,
\end{equation}
where we write $v\sim u$ if $uv\in E$. 
Note that $\Delta_{(m_0,m_1)}$ has only real and nonnegative eigenvalues, always has eigenvalue $0$,
and the corresponding eigenspace consists precisely of constant functions on $V$ if $E'=\{uv\in E \mid m_1(uv)>0\}$ 
is connected, which we assume from here on. 
The second smallest eigenvalue of $\Delta_{(m_0,m_1)}$, which is positive, will be denoted by $\lambda_1(G, (m_0,m_1))$
and referred to as the {\em first nonzero eigenvalue} of $\Delta_{(m_0,m_1)}$. 
It is a standard fact that %at least when $m_1$ is strictly positive on $E$,
$\lambda_1(G, (m_0,m_1))$ is characterized variationally as
\begin{eqnarray}\label{lambda1-characterization1}
\lambda_1(G, (m_0,m_1)) &=& 
\min_{\varphi} \frac{\langle \Delta_{(m_0,m_1)} \varphi, \varphi \rangle}{
\langle \varphi, \varphi \rangle} \nonumber\\
&=& 
\min_{\varphi} \frac{\sum_{uv\in E} m_1(uv)
( \varphi(u) - \varphi(v) )^2}{\sum_{u\in V} m_0(u) %( 
\varphi(u)
%- \overline{\varphi} )
^2},
\end{eqnarray}
where %$\overline{\varphi} = \frac{1}{\sum_{u\in V} m_0(u)} \sum_{u\in V} m_0(u) \varphi(u)$ and 
the minimum is taken over all nonzero %nonconstant 
functions $\varphi$ such that\\ $\sum_{u\in V} m_0(u) \varphi(u)=0$, meaning that 
$\varphi$ is orthogonal to constant functions, that is, eigenfunctions of the eigenvalue $0$.
%In fact, \eqref{lambda1-characterization1} still holds if one uses $\R^N$-valued maps instead of $\R$-valued 
%functions, that is,
%\begin{equation}\label{lambda1-characterization2}
%\lambda_1(G, (m_0,m_1)) = \inf_{\bm{\varphi}} \frac{\sum_{uv\in E} m_1(uv)
%\| \bm{\varphi}(u) - \bm{\varphi}(v) \|^2}{\sum_{u\in V} m_0(u) \|\bm{\varphi}(u)- \overline{\bm{\varphi}} \|^2},
%\end{equation}
%where the infimum is taken over all nonconstant maps $\bm{\varphi}\colon V\to \R^N$.

Other than the operator language, we can employ the matrix language to describe the graph Laplacian.
To do so, we write $V=\{ 1,\dots, |V| \}$ and choose the orthonormal basis
$$
e_i\colon j\in V \mapsto \frac{\delta_{ij}}{\sqrt{m_0(i)}}\in \R,\quad i\in V, 
$$
of $\R^V$. 
Then the corresponding representation matrix $L := L_{(m_0,m_1)}$ of the Laplacian $\Delta_{(m_0,m_1)}$
is given by $L= D^{-1/2} L_0 D^{-1/2}$, where $D=\mathrm{diag}(m_0(1),\dots, m_0(|V|))$ and $L_0$ 
has diagonal components
$$
(L_0)_{ii}=\sum_{j\sim i} m_1(ij),\quad i\in V
$$
and off-diagonal components
$$
(L_0)_{ij}= \left\{ 
\renewcommand{\arraystretch}{1.5}
\begin{array}{cc} - m_1(ij), 
& %i\neq j\in V, 
ij\in E,\\ 
0, & ij\notin E. 
\end{array}
\right. 
$$
The matrix $L$ is positive, symmetric, and has eigenvalue $0$ with eigenvector 
$x_0 = \bigl(\sqrt{m_0(1)},\dots,\sqrt{m_0(|V|)}\bigr)$. 
Note also that the variational characterization \eqref{lambda1-characterization1} of 
$\lambda_1(G, (m_0,m_1))$ is expressed as 
\begin{equation*}%\label{lambda1-characterization2}
\lambda_1(G, (m_0,m_1)) = \min_{x} \frac{Lx\cdot x}{x\cdot x}, 
\end{equation*}
where $\cdot$ denotes the Euclidean inner product on $\R^{|V|}$ and the minimum is taken over 
all nonzero vectors $x\in \R^{|V|}$ such that $x\cdot x_0=0$.

\begin{Remark}\label{fvcorr}
Associated with the operator-matrix correspondence above, it should be noted that given 
a vector $x=(x_1,\dots,x_n)\in \R^{|V|}$, the corresponding function $\varphi\in \R^V$ 
is given by 
$$
\varphi(i) = x_i/\sqrt{m_0(i)},\quad 1\leq i\leq |V|. 
$$
If $\varphi, \psi\in \R^V$ correspond to $x,y\in \R^{|V|}$, respectively, then 
$\langle \varphi, \psi \rangle = x\cdot y$. 
This observation is necessary in working with explicit examples in Section 3. 
\end{Remark}

We now formulate a first-eigenvalue maximization problem whose variable is the edge-length function 
$l\colon E \to \R_{>0}$. 
Following Fujiwara \cite{Fujiwara}, for a given $l$, we define a vertex-weight $m_0\colon V \to \R_{>0}$ 
and an edge-weight $m_1 \colon E \to \R_{>0}$ by 
$$
m_0(u) = \sum_{v \sim u} l(uv),\,\, u \in V\quad \mbox{and}\quad 
m_1(uv) = \frac{1}{l(uv)},\,\, uv \in E, 
$$
respectively. 
Let $\Delta_l$ denote the corresponding Laplacian. 

\begin{Problem}\label{maxspecGN}
Over all edge-length functions $l$, subject to the normalization 
\begin{equation}\label{constraint}
\sum_{u \in V} m_0(u) = 2 \sum_{uv\in E} l(uv) = 1, 
\end{equation}
maximize the first nonzero eigenvalue $\lambda_1(G,l)$ of $\Delta_{l}$.
\end{Problem}

\begin{Remark} 
In the definition of the weighted Laplacian, it is standard to assume that 
$m_0(u)=\sum_{v\sim u} m_1(u,v)$ (cf.~\cite{Chung}). 
The above definition of $\Delta_l$ by Fujiwara deviates from this convention. 
However, it has a merit (cf.~\cite{Fujiwara}): the eigenvalues of the Laplacian 
of a compact Riemannian manifold $(M,g)$ are approximated in a certain coarse sense 
by the eigenvalues of Fujiwara's Laplacian of  $1/n$-nets in $M$ equipped with 
appropriate graph structures and edge-length functions $l_n\equiv 1/n$, as 
$n\to \infty$.\\ 
\indent 
The squared edge-length function $l^2$ is regarded as a graph analogue of the 
Riemannian metric. 
For $c>0$, we have $\Delta_{cl} = \frac{1}{c^2} \Delta_l$, which is consistent 
with the relation $\Delta_{c^2g} = \frac{1}{c^2} \Delta_g$ for the Reimannian 
Laplacian under the rescaling of Riemannian metric $g$. 
\end{Remark}

Following El Soufi-Ilias \cite{ElSoufiIlias} (see also \cite{Nadirashvili}), 
we introduce the following definition: 
\begin{Definition}
An edge-length function $l$ satisfying the constraint \eqref{constraint} 
is said to be an {\em extremal solution} of Problem \ref{maxspecGN} if for any 
analytic curve $l\colon I \to (\R_{>0})^E$, where $0\in I\subset \R$ is a small interval 
and $(\R_{>0})^E$ is the set of positive functions on $E$, 
such that $l(0)=l$ and $l(t)$ satisfies \eqref{constraint} for all $t\in I$, 
the left and right derivatives of $\lambda_1(l(t))$ at $t=0$ satisfy
\begin{equation}\label{extremal}
\frac{d}{dt} \lambda_1(l(t))\big|_{+0}\leq 0\leq \frac{d}{dt} \lambda_1(l(t))\big|_{-0}. 
\end{equation}
%$\lambda_1(l(t))\leq \lambda_1(l(0))+o(t)$ as $t\to 0$. 
%, or $\lambda_1(l(t)\geq \lambda_1(l(0))+0(t)$ as $t\to 0$. \lambda_1では不要 
\end{Definition}

\begin{Theorem}\label{thm-embed-var-1} 
Let $l$ be an extremal solution of Problem \ref{maxspecGN}. 
Then there exist eigenfunctions $\varphi_1, \cdots, \varphi_N$ of the eigenvalue $\lambda_1(G,l)$ of $\Delta_{l}$ 
so that the map $\varphi = (\varphi_1, \cdots, \varphi_N)\colon V\to \R^N$ satisfies 
%\begin{equation}\label{embedlength-var-1}
%\sum_{1 \leq i \leq d} \Phi_{\varphi_i}(uv) = 1, \quad \forall uv \in E.
%\end{equation}
\begin{equation}\label{lipschitz}
l(uv)^2 = \frac{\| \varphi(u) - \varphi(v) \|^2}{1-\lambda_1(G,l) 
\left( \| \varphi(u) \|^2 + \| \varphi(v) \|^2 \right)} 
\end{equation}
for all $uv \in E$, where $\| \cdot \|$ is the Euclidean norm of $\R^N$. 
\end{Theorem}

\begin{Remark} 
(i)\,\, Note that we may assume that $N$, the dimension of the ambient Euclidean space, 
is at most the multiplicity, that is, the dimension of the eigenspace, of $\lambda_1(G,l)$. 
(Otherwise, $\varphi_1,\dots,\varphi_N$ should satisfy linear relations and then the image 
of the map $\varphi$ lies in a proper subspace of $\R^N$.) 
Since the multiplicity of $\lambda_1(G,l)$ is less than or equal to $|V|-1$, we obtain 
$N\leq |V|-1$.

\smallskip\noindent
(ii)\,\, Given any positive constant $C$, the map $\widetilde{\varphi}=\sqrt{C}\, \varphi$, 
where $\varphi$ is as in the theorem, satisfies 
$$
l(uv)^2 = \frac{\| \widetilde{\varphi}(u) - \widetilde{\varphi}(v) \|^2}{C-\lambda_1(G,l) 
\left( \| \widetilde{\varphi}(u) \|^2 + \| \widetilde{\varphi}(v) \|^2 \right)} 
$$
for all $uv \in E$. 

\smallskip\noindent
(iii)\,\, In the theorem, if the function $uv\in E\mapsto \| \varphi(u) \|^2 + \| \varphi(v) \|^2\in \R$ 
%$\| \varphi(\cdot) \|\colon V\to \R$ 
is constant ($:= C'$), then the map $\widehat{\varphi} = \varphi/\sqrt{1-\lambda_1(G,l)C'}$ 
satisfies $l(uv) = \| \widehat{\varphi}(u) - \widehat{\varphi}(v) \|$ for all $uv \in E$, that is, 
$\widehat{\varphi}$ is an isometric map with respect to $l$, meaning that 
the pull-back of the Euclidean distance by $\widehat{\varphi}$ coincides with $l$. 
This is, however, not possible in general. See Example \ref{example-path}. 

\smallskip\noindent
(iv)\,\, When the vertex set of a graph is equipped with a metric, the problem 
of finding a bi-Lipschitz embedding with small distortion into a Euclidean space has been 
extensively studied (cf.~\cite{LinialLondonRabinovich}).
In this context, we observe that \eqref{lipschitz} implies the existence of $D>0$ such that 
$$
\frac{1}{D} l(uv)\leq \|\varphi(u)-\varphi(v)\|\leq l(uv)\quad \mbox{for all $uv\in E$}, 
$$
meaning that the map $\varphi$ is locally bi-Lipschitz. 
Additionally, by the triangle inequality, we can derive 
$$
\|\varphi(u)-\varphi(v)\|\leq d_l(u, v)\quad \mbox{for all $u, v\in V$}, 
$$
where $d_l$ is the metric on $V$ defined by 
$
d_l(u,v) = \min \sum_{i=1}^s l(u_{i-1}u_i),
$
with the minimum taken over all paths $u_0u_1,u_1u_2,\dots,u_{s-1}u_s$ joining 
$u_0=u$ and $u_s=v$.
Thus, $\varphi$ is globally $1$-Lipschitz.
\end{Remark}

\section{Proof of Theorem \ref{thm-embed-var-1}}
%Variational formula for the first eigenvalue}

First we consider deformations of vertex-weight $m_0$ and edge-weight $m_1$ depending analytically on the deformation parameter, 
and compute the derivatives of analytic branches of the first nonzero eigenvalue of the Laplacian. 
Fix $m_0$ and $m_1$, and assume $m_1(uv)>0$ for all $uv\in E$. 
Consider analytic curves 
$$
m_0\colon I \to (\R_{>0})^V,\quad m_1\colon I \to (\R_{>0})^E
$$
such that $m_0(0)=m_0$ and $m_1(0)=m_1$. %, where $0\in I\subset \R$ is a small interval. 
Let $\mu$ denote the multiplicity of the first nonzero eigenvalue $\lambda_1(G, m_0,m_1)$ of the Laplacian 
$\Delta_{(m_0,m_1)}$. 
Then, by a classical result of Rellich (see \cite[p.~33, Theorem 1]{Rellich}), there are a neighborhood 
$I'\subset I$ of $0$, analytic functions $\lambda^{(i)}_1 \colon I' \to \R_{>0}$ and analytic maps 
$\varphi^{(i)} \colon I' \to \R^V$, $1\leq i\leq \mu$, satisfying 
\begin{enumerate}
\renewcommand{\theenumi}{\roman{enumi}}
\renewcommand{\labelenumi}{(\theenumi)}
\item $\lambda^{(i)}_1(0) = \lambda_1(G, m_0,m_1)$,\quad $1 \leq i \leq \mu$, 
\item $\Delta_{(m_0(t),m_1(t))} \varphi^{(i)}(t) = \lambda^{(i)}_1(t)\, \varphi^{(i)}(t)$,
\quad $t\in I'$, \ $1 \leq i \leq \mu$, 
\item $\langle \varphi^{(i)}(t), \varphi^{(j)}(t) \rangle_{t} = \delta_{ij}$, 
\quad $t\in I'$, \ $1 \leq i, j \leq \mu$,
\end{enumerate}
where $\langle \cdot, \cdot \rangle_t$ is the inner product on $\R^V$ with respect to the 
vertex-weight $m_0(t)$.

\begin{Remark}
In fact, Rellich considered a matrix whose components are complex-analytic functions of complex variable $t$ 
and which is hermitian for real $t$. 
His result can be applied to our situation since the representation matrix (see Section 1) 
of $\Delta_{(m_0(t),m_1(t))}$ can be extended to such a matrix. 
(This is simply because any real-analytic function can be extended to a complex-analytic function.) 
Our situation is special in that the extended matrix is real (therefore, symmetric since it is hermitian) 
for real $t$. 
\end{Remark}

We will calculate $\dot{\lambda}^{(i)}_1 := \frac{d}{dt}\lambda^{(i)}_1(t)|_{t=0}$. 
We also denote $\varphi^{(i)} := \varphi^{(i)}(0)$ and $\dot{\varphi}^{(i)} := 
\frac{d}{dt}\varphi^{(i)}(t)|_{t=0}$. 
First, by differentiating both sides of 
$\Delta_{(m_0(t),m_1(t))} \varphi^{(i)}(t) = \lambda^{(i)}_1(t) \varphi^{(i)}(t)$ at $t=0$, we obtain 
\begin{equation*}
\left.\frac{d}{dt} \Delta_{(m_0(t),m_1(t))} \right|_{t=0} \varphi^{(i)} + \Delta_{(m_0,m_1)} \dot{\varphi}^{(i)} 
= \dot{\lambda}^{(i)}_1 \varphi^{(i)} + \lambda_1(G, m_0,m_1) \dot{\varphi}^{(i)}.
\end{equation*}
By taking the $m_0$-inner product with $\varphi^{(i)}$, we obtain 
\begin{equation*}
\left\langle \left.\frac{d}{dt} \Delta_{(m_0(t),m_1(t))} \right|_{t=0} \varphi^{(i)}, \varphi^{(i)} \right\rangle 
+ \langle \Delta_{(m_0,m_1)} \dot{\varphi}^{(i)}, \varphi^{(i)} \rangle 
= \dot{\lambda}^{(i)}_1 + \lambda_1(G, m_0,m_1) \langle \dot{\varphi}^{(i)}, \varphi^{(i)} \rangle. 
\end{equation*}
Since $\Delta_{(m_0,m_1)}$ is a symmetric operator, 
$$
\langle \Delta_{(m_0,m_1)} \dot{\varphi}^{(i)}, \varphi^{(i)} \rangle 
= \langle \dot{\varphi}^{(i)}, \Delta_{(m_0,m_1)} \varphi^{(i)} \rangle 
= \lambda_1(G, m_0,m_1) \langle \dot{\varphi}^{(i)}, \varphi^{(i)} \rangle. 
$$
Thus, 
$$
\dot{\lambda}^{(i)}_1 = \left\langle \left.\frac{d}{dt} \Delta_{(m_0(t),m_1(t))} \right|_{t=0} \varphi^{(i)}, 
\varphi^{(i)} \right\rangle.
$$
The differential of the Laplacian is 
\begin{eqnarray*}
\lefteqn{\left( \left.\frac{d}{dt} \Delta_{(m_0(t),m_1(t))} \right|_{t=0} \varphi \right) (u)}\\  
&=& \sum_{v\sim u} \frac{\dot{m}_1(uv)}{m_0(u)} \left( \varphi(u) - \varphi(v) \right) 
- \sum_{v\sim u} \frac{m_1(uv) \dot{m}_0(u)}{m_0(u)^2} \left( \varphi(u) - \varphi(v) \right),
\quad u\in V.
\end{eqnarray*}
Therefore, we obtain 
\begin{eqnarray}\label{lambda1-derivative}
\dot{\lambda}^{(i)}_1 
&=& \sum_{u \in V} \sum_{v\sim u} \dot{m}_1(uv) \left( \varphi^{(i)}(u) - \varphi^{(i)}(v) \right) \varphi^{(i)}(u) 
\nonumber\\
&& - \sum_{u \in V} \sum_{v\sim u} \frac{m_1(uv) \dot{m}_0(u)}{m_0(u)} \left( \varphi^{(i)}(u) - \varphi^{(i)}(v) \right)\varphi^{(i)}(u)\\ 
&=& \sum_{uv \in E} \dot{m}_1(uv) \left( \varphi^{(i)}(u) - \varphi^{(i)}(v) \right)^2 
- \lambda_1(G, m_0,m_1) \sum_{u \in V} \dot{m}_0(u)\, \varphi^{(i)}(u)^2. \nonumber
\end{eqnarray}

\medskip\noindent
{\em Proof of Theorem \ref{thm-embed-var-1}.}\quad 
The proof follows closely El Soufi-Ilias' proof \cite{ElSoufiIlias} of the Nadirashvili theorem \cite{Nadirashvili} 
mentioned in Introduction. 

Let $l\colon I \to (\R_{>0})^E$ be an analytic curve such that $l(0)=l$. 
Plugging 
$$
\dot{m}_0(u) = \sum_{v\sim u} \dot{l}(uv)\quad \mbox{and}\quad \dot{m}_1(uv) = - l(uv)^{-2}\dot{l}(uv) 
$$
into \eqref{lambda1-derivative}, we obtain 
\begin{eqnarray*}
\dot{\lambda}^{(i)}_1 
&=& - \sum_{uv \in E} l(uv)^{-2}\dot{l}(uv) \left( \varphi^{(i)}(u) - \varphi^{(i)}(v) \right)^2 
- \lambda_1(G,l) \sum_{u \in V} \sum_{v\sim u} \dot{l}(uv) \varphi^{(i)}(u)^2\\ 
&=& - \sum_{uv \in E} \dot{l}(uv) \left[ l(uv)^{-2} \left( \varphi^{(i)}(u) - \varphi^{(i)}(v) 
\right)^2 + \lambda_1(G,l) \left( \varphi^{(i)}(u)^2 + \varphi^{(i)}(v)^2 \right) \right]. 
\end{eqnarray*}

For $\varphi\in \R^V$, we define a function $q_\varphi\colon E\to \R$ by 
$$
q_\varphi(uv) = l(uv)^{-2} \left( \varphi(u) - \varphi(v) \right)^2 + \lambda_1(G,l) 
\left( \varphi(u)^2 + \varphi(v)^2 \right), 
$$
so that 
$$
\dot{\lambda}_1^{(i)} = - \sum_{uv\in E} \dot{l}(uv) q_{\varphi^{(i)}}(uv). 
$$

\medskip\noindent
{\bf Claim 1.}\quad 
For any function $\rho\colon E \to \R$ satisfying $\sum_{uv\in E} \rho(uv) = 0$, 
there exists a first eigenfunction $\varphi$ of the Laplacian $\Delta_l$ such that
\begin{equation*}%\label{claim}
\sum_{uv\in E} \rho(uv)\, q_{\varphi}(uv) = 0.
\end{equation*} 

\medskip
Choose the curve $l(t)$ so that $\dot{l} = \rho$ and $l(t)$ satisfies the constraint 
\eqref{constraint} for all $t\in I$. 
(In fact, it suffices to set $l(t) = l + t \rho$.)
Since $l=l(0)$ is an extremal solution, we must have 
\begin{eqnarray*}
&& \frac{d}{dt} \lambda_1(G,l(t))\big|_{-0} 
= \max_{1\leq i\leq \mu} \dot{\lambda}^{(i)}_1 
= - \min_{1\leq i\leq \mu} \sum_{uv \in E} \dot{l}(uv)\, q_{\varphi^{(i)}}(uv) \geq 0
\quad \mbox{and}\quad \\
&& \frac{d}{dt} \lambda_1(G,l(t))\big|_{+0} 
= \min_{1\leq i\leq \mu} \dot{\lambda}^{(i)}_1 
= - \max_{1\leq i\leq \mu}\sum_{uv \in E} \dot{l}(uv)\, q_{\varphi^{(i)}}(uv) \leq 0. 
\end{eqnarray*}
Thus one finds $1\leq i_1, i_2\leq \mu$ such that 
$$
\sum_{uv \in E} \rho(uv)\, q_{\varphi^{(i_1)}}(uv)\leq 0\leq 
\sum_{uv \in E} \rho(uv)\, q_{\varphi^{(i_2)}}(uv). 
$$
Now suppose $\mu=\dim E_1(l)\geq 2$, where $E_1(l)$ is the $\lambda_1(G,l)$-eigenspace 
of $\Delta_l$. 
Then one can choose a continuous curve $\varphi\colon [0,1]\to E_1(l)\setminus \{0\}$ 
so that $\varphi(0)=\varphi^{(i_1)}$, $\varphi(1)=\varphi^{(i_2)}$. 
By the intermediate value theorem, one finds $t_0\in [0,1]$ such that 
$$
\sum_{uv \in E} \rho(uv)\, q_{\varphi(t_0)}(uv) = 0.  
$$
If $\mu=1$, then $\lambda^{(1)}(t) = \lambda_1(G, l(t))$ is analytic in $t$ and the 
extremality of $l$ is equivalent to $\frac{d}{dt} \lambda_1(G,l(t))=0$, which implies 
$$
\sum_{uv \in E} \rho(uv)\, q_{\varphi^{(1)}}(uv) = 0.  
$$

\medskip
Let $\mathcal{C}$ be the convex hull 
of the set $\{ q_{\varphi} \mid \varphi\in E_1(l) \}$ in $\R^E \cong \R^{|E|}$. 
%\{ \sum_{1 \leq i \leq d} q_{\varphi_i} \mid \varphi_i \in E_1(l), d \in \N \} 
Note that $\mathcal{C}$ is a closed convex cone. 

\medskip\noindent
{\bf Claim 2.}\quad The constant function ${\bf 1}\in \mathcal{C}$. 

\medskip
Suppose that ${\bf 1}\notin \mathcal{C}$. 
Then by the hyperplane separation theorem, there exists a hyperplane of $\R^E$ with a normal vector $\nu$ such that 
\begin{align}
& \langle\langle {\bf 1}, \nu \rangle\rangle = \sum_{uv\in E} \nu(uv) > 0, \label{separation1}\\
& \langle\langle q_{\varphi}, \nu \rangle\rangle = \sum_{uv\in E} q_{\varphi}(uv)\, \nu(uv) \leq 0,\quad \forall \varphi\in E_1(l), 
\label{separation2}
\end{align}
where $\langle\langle \cdot, \cdot \rangle\rangle$ denotes the canonical inner product
on $\R^E$. The function 
$\dps \rho = \nu - \frac{\langle\langle \nu, {\bf 1}\rangle\rangle}{\langle\langle {\bf 1}, {\bf 1} \rangle\rangle} {\bf 1}$ 
satisfies $\sum_{uv\in E} \rho(uv) = 0$. 
On the other hand, for any $\varphi\in E_1(l)\setminus \{0\}$, we have 
$$
\langle\langle \rho, q_{\varphi} \rangle\rangle = \langle\langle \nu, q_{\varphi} \rangle\rangle 
- \frac{\langle\langle \nu, {\bf 1}\rangle\rangle}{\langle\langle {\bf 1}, {\bf 1} \rangle\rangle} \langle\langle {\bf 1}, q_{\varphi} \rangle\rangle 
< 0
$$
by \eqref{separation1}, \eqref{separation2} and the fact that $q_\varphi$ is nonnegative and positive somewhere. 
This contradicts Claim 1. 

\medskip 
By Claim 2, there exist $\varphi_k\in E_1(l)$, $1\leq k\leq N$, such that 
\begin{eqnarray*}
1 &=& \sum_{k=1}^N q_{\varphi_k} (uv)\\ 
&=& l(uv)^{-2} \|\varphi(u) - \varphi(v) \|^2 
+ \lambda_1(G,l) \left( \|\varphi(u)\|^2 + \|\varphi(v)\|^2 \right),\quad uv\in E, 
\end{eqnarray*}
where $\varphi=(\varphi_1,\dots,\varphi_N)$. 
Slight rearrangement gives the conclusion of Theorem \ref{thm-embed-var-1}. 
\hfill{$\square$}

\medskip
The following proposition is a weak converse to Theorem \ref{thm-embed-var-1}. 

\begin{Proposition}
Let $l$ be an edge-length function satisfying the constraint \eqref{constraint}, and suppose that 
there exist a positive constant $C$ and an orthonormal basis $\{ \varphi_1, \cdots, \varphi_N \}$ 
of $E_1(l)$, the $\lambda_1(G,l)$-eigenspace of $\Delta_l$, so that the map 
$\varphi = (\varphi_1, \cdots, \varphi_N) : V\to \R^N$ satisfies 
%\begin{equation}\label{embedlength-var-1}
%\sum_{1 \leq i \leq d} \Phi_{\varphi_i}(uv) = 1, \quad \forall uv \in E.
%\end{equation}
\begin{equation}%\label{lipschitz}
l(uv)^2 = \frac{\| \varphi(u) - \varphi(v) \|^2}{C-\lambda_1(G,l) 
\left( \| \varphi(u) \|^2 + \| \varphi(v) \|^2 \right)} 
\end{equation}
for all $uv \in E$. 
Then $l$ is an extremal solution of Problem \ref{maxspecGN}. 
\end{Proposition}

\begin{proof}
Let $l\colon I \to (\R_{>0})^E$ be an analytic curve such that $l(0)=l$ and $l(t)$ satisfies the constraint 
\eqref{constraint} for all $t\in I$, and set $\rho = \left. \frac{d}{dt} l(t) \right|_{t=0}$. 
Using the same notations $\lambda^{(i)}_1$ and $\varphi^{(i)}$ as in the proof of Theorem \ref{thm-embed-var-1}, 
we have for small $t$: 
\begin{equation*}%\label{specequal}
\sum_{i=1}^N \lambda_i(l(t)) = \sum_{i=1}^N \lambda^{(i)}_1(t). 
\end{equation*}
It follows that the left-hand side is differentiable at $t=0$, and 
\begin{eqnarray*}
&& \left.\frac{d}{dt} \sum_{i=1}^N \lambda_i(l(t)) \right|_{t=0} 
%&=& \left.\frac{d}{dt} \sum_{i=1}^N \lambda^{(i)}_1(t) \right|_{t=0} \\
\,\,=\,\, \sum_{i=1}^N \dot{\lambda}^{(i)}_1\\
&=& -\sum_{i=1}^N \sum_{uv\in E} \rho(uv) \left[ l(uv)^{-2} \left( \varphi^{(i)}(u) - \varphi^{(i)}(v) \right)^2 + \lambda_1(G,l) 
\left( \varphi^{(i)}(u)^2 + \varphi^{(i)}(v)^2 \right) \right] \\
&=& -\sum_{uv\in E} \rho(uv) \left[ l(uv)^{-2} \| \varphi(u) - \varphi(v) \|^2 + \lambda_1(G,l) 
\left( \| \varphi(u) \|^2 + \| \varphi(v) \|^2 \right) \right] \\
&=& -C \sum_{uv\in E} \rho(uv)\,\, =\,\, 0.
\end{eqnarray*} 
Therefore, 
$$
\frac{d}{dt} \lambda_1(G,l(t))\big|_{-0} = \max_{1\leq i\leq \mu} \dot{\lambda}^{(i)}_1 
\quad \mbox{and}\quad 
\frac{d}{dt} \lambda_1(G,l(t))\big|_{+0} = \min_{1\leq i\leq \mu} \dot{\lambda}^{(i)}_1 
$$
satisfy \eqref{extremal}.
%Since the inequality $\lambda_1(l(t))\leq \frac{1}{N} \sum_{i=1}^N \lambda^{(i)}_1(t)$ holds 
%for small $t$, we obtain $\lambda_1(l(t))\leq \lambda_1(l(0)) + o(t)$, that is, $l=l(0)$ is extremal.  
\end{proof}

\section{Examples}

In this section, we consider graphs of small size and find extremal solutions of Problem \ref{maxspecGN} and the corresponding eigenfunctions satisfying \eqref{lipschitz} by numerical means using Mathematica. 
In fact, we treat all the graphs with $|V|\leq 4$ and $|E|\leq 4$ (other than the trivial one 
consisting of two vertices and one edge).

\begin{Example}\label{example-path}
Let $P_3$ be the path graph consisting of three vertices $1$, $2$, $3$ and two edges $12$, $23$.
We assign length parameters $a>0$ and $b>0$ to the edges $12$ and $23$, respectively. 
Then the weights of the vertices $1$, $2$ and $3$ are $a$, $a+b$ and $b$, respectively. 
By the constraint $2(a+b) = 1$, we have $b=(1-2a)/2$.
Thus, the eigenvalue maximization problem is a problem of the single variable $a$ constrained 
in the range $0<a<1/2$. 
The representation matrix of the Laplacian reads
$$
\begin{pmatrix}
a^{-2} & -a^{-3/2}(a+b)^{-1/2} & 0 \\ 
-a^{-3/2}(a+b)^{-1/2} & (a^{-1}+b^{-1})/(a+b) & -b^{-3/2}(a+b)^{-1/2} \\
0 & -b^{-3/2}(a+b)^{-1/2} & b^{-2}
\end{pmatrix}
$$
with $b=(1-2a)/2$.
The first nonzero eigenvalue can be explicitly computed as 
$$
\lambda_1(P_3, (a,b)) = \frac{4a^2-2a+1 -\sqrt{-112a^4+112a^3-20a^2-4a+1}}{2a^2(2a-1)^2}. 
$$
It takes its maximum value $16$ at $a=1/4$ (and $b=1/4$). 
The eigenvalue $16$ is simple with an eigen{\em vector} $(-1, 0 , 1)$ of the {\em matrix}. 
The corresponding eigen{\em function} of the Laplacian {\em operator} is 
$(\varphi(1), \varphi(2), \varphi(3))= (-2,0,2)$ (cf.~Remark \ref{fvcorr}),
which realizes $P_3$ on the real line. 
If we choose $\varphi' = 1/(8\sqrt{2}) \varphi$, then it satisfies \eqref{lipschitz}.   

For the path graph $P_4$ consisting of four vertices $1$, $2$, $3$ and $4$ and three edges $12$, $23$, $34$, 
we assign edge length parameters $a$, $b$ and $c$ to the edges $12$, $23$ and $34$, respectively.
Then the eigenvalue maximization problem is a problem of the two variables $a$ and $c$ constrained in the range $a>0$, $c>0$ and $a+c<1/2$. %because of $0<b=(1-2a-2c)/2$. 
The first nonzero eigenvalue has the maximum value $\simeq 18.6694$ at $a=c \simeq 0.1905$ ($b \simeq 0.1190$).
Its multiplicity is one, and the first eigenfunction 
$$
(\varphi(1), \varphi(2), \varphi(3), \varphi(4))\simeq (0.1734, 0.0559, -0.0559, -0.1734)
$$
realizes $P_4$ on the real line and satisfies \eqref{lipschitz}. 
It is worthwhile to mention that no scalar multiple of $\varphi$ can be isometric, because 
$|\varphi(1)-\varphi(2)|/|\varphi(2)-\varphi(3)|\neq a/b$.
\end{Example}

\begin{Example}
Let $K_{1,3}$ denote the complete bipartite graph consisting of the first vertex $1$ and the second ones $2,3,4$.
Assigning length parameters $a$, $b$ and $c$ to the edges $12$, $13$ and $14$, respectively, 
the eigenvalue maximization problem is a problem of the two variables $a$ and $b$ constrained in the range $a>0$, $b>0$ and $a+b<1/2$.
The first nonzero eigenvalue has a maximum $36$ at $a=b=1/6$ ($=c$), and its multiplicity is two 
with orthonormal eigenvectors $(1/\sqrt{2})(0, -1, 0 , 1)$, $(1/\sqrt{6})(0, -1, 2 , -1)$.  
By the eigen-map $\varphi=(\varphi_1, \varphi_2)$ consisting of orthonormal eigenfunctions 
\begin{align*}
& (\varphi_1(1), \varphi_1(2), \varphi_1(3), \varphi_1(4)) 
= \sqrt{3} (0, -1, 0 , 1),\\ 
& (\varphi_2(1), \varphi_2(2), \varphi_2(3), \varphi_2(4)) = (0, -1, 2 , -1), 
\end{align*}
$K_{1,3}$ is realized as a symmetric tripod in $\R^2$, and 
$1/(12\sqrt{2})\, \varphi$ satisfies \eqref{lipschitz}. 
\end{Example} 

For all the graphs of the following examples, the first nonzero eigenvalue diverges to positive infinity at boundary points of the constraint region and therefore, maximizing solutions of Problem \ref{maxspecGN} do not exist. However, there may exist locally maximizing or extremal solutions. 

\begin{Example}
Let $C_3$ be the cycle graph consisting of three vertices $1$, $2$, $3$ and three edges $12$, $23$, $31$.
By assigning length parameters $a$, $b$ and $c$ to the edges $12$, $23$ and $31$, respectively,
%the weights of the vertices $1$, $2$ and $3$ are $c+a$, $a+b$ and $b+c$, respectively. 
%By the constraint $2(a+b+c)=1$, 
the eigenvalue maximization problem is a problem of the two variables $a$ and $b$ constrained in the range $a>0$, $b>0$ and $a+b<1/2$. 
Plot the first nonzero eigenvalue over the $a$-$b$ plane as in Figure \ref{fig:C3}.
Then it is observed that the first nonzero eigenvalue diverges to positive infinity as two of the parameters $a$, $b$ and $c$ approach $0$ (the remaining parameter approaches $1/2$). 
However, it has a local maximum $54$ at $a=b=1/6$ ($=c$), and its multiplicity is two 
with orthonormal eigenvectors $(1/\sqrt{2})(-1, 0 , 1)$, $(1/\sqrt{6})(-1, 2 , -1)$. 
By the eigen-map $\varphi=(\varphi_1, \varphi_2)$ consisting of orthonormal eigenfunctions 
\begin{align*}
& (\varphi_1(1), \varphi_1(2), \varphi_1(3)) = \sqrt{3/2}(-1, 0 , 1),\\ 
& (\varphi_2(1), \varphi_2(2), \varphi_2(3)) = (1/\sqrt{2})(-1, 2 , -1), 
\end{align*}
$C_3$ is realized as a regular triangle, and $(1/12\sqrt{3})\, \varphi$ satisfies \eqref{lipschitz}. 

There are also three saddle-type (hence, extremal) solutions; One of them is 
$(a,b,c) = ((3-\sqrt{3})/12,(3-\sqrt{3})/12,1/(2\sqrt{3}))$ and the other two are permutations of this. 
%$(1/(2\sqrt{3}),(3-\sqrt{3})/12,(3-\sqrt{3})/12)$ ,$((3-\sqrt{3})/12,1/(2\sqrt{3}),(3-\sqrt{3})/12)$.
For all these solutions, the first nonzero eigenvalue is simple and $\simeq 41.5692$. 
For the above $(a,b,c)$, %= ((3-\sqrt{3})/12,(3-\sqrt{3})/12,1/(2\sqrt{3}))$, 
the first eigenfunctions are scalar multiples of $(-1,0,1)$ which realize $C_3$ on the real line. 
If we choose $(\varphi(1), \varphi(2), \varphi(3))\simeq 0.0873 (-1,0,1)$, then it satisfies \eqref{lipschitz}. 
However, no scalar multiple of $(-1,0,1)$ can be isometric.  

\vskip\baselineskip
\begin{figure}[htbp]
\begin{center}
\includegraphics[bb=0 0 384 362,scale=0.4]{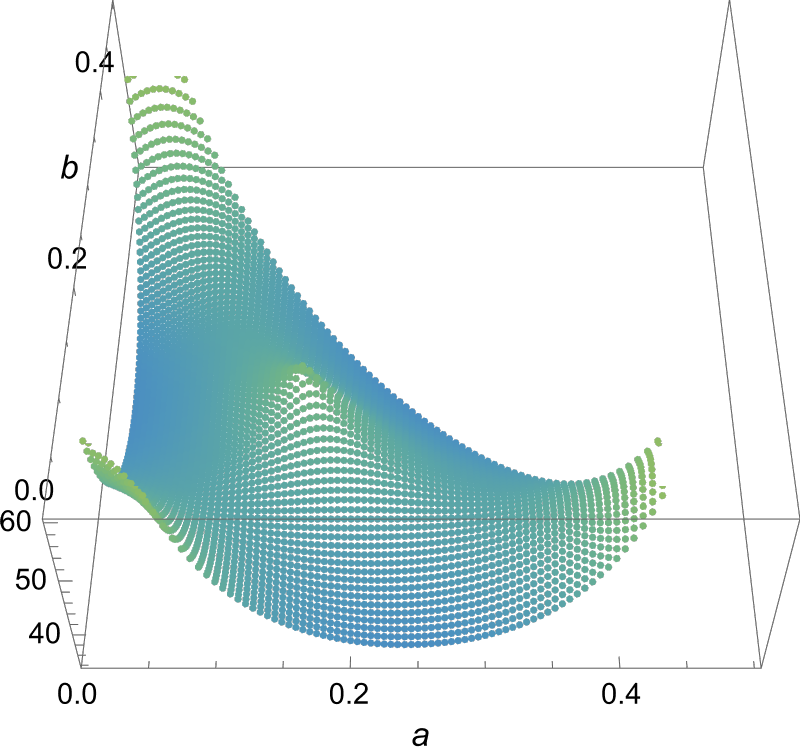}
\end{center}
\caption{Plot of the first nonzero eigenvalue on $a$-$b$ plane for $C_3$.}
\label{fig:C3}
\end{figure}    

For the $4$-cycle graph $C_4$ with length parameters $a$, $b$, $c$ and $d$ assigned similarly, the eigenvalue maximization problem is a problem of the three variables $a$, $b$ and $c$ constrained in the range $a>0$, $b>0$, $c>0$ and $a+b+c<1/2$. 
To plot the first nonzero eigenvalue in 3D, we fix one parameter, say $a$.  
Let $a_t=(1/2)(t/n) \in (0,1/2)$ for $t=1, \cdots ,n-1$.
For each $t$, we observe the behaviour of the first nonzero eigenvalue as $b$ and $c$ vary.
Then, for each $t$, the first nonzero eigenvalue has a local maximum on the segment $b=c$. (See Figure \ref{fig:C4at} in which $n=11$.) 
This implies that when the whole parameters $a$, $b$ and $c$ vary, one finds a local maximum point on the section $b=c$. 
So let $b=c$ and plot the first nonzero eigenvalue over the $a$-$b$ plane as in Figure \ref{fig:C4b=c}. 
Then it has a local maximum $64$ at $a=b=1/8$ ($=c=d$), and its multiplicity is two 
with orthonormal eigenvectors $(1/\sqrt{2})(-1, 0 , 1,0)$, $(1/\sqrt{2})(0, -1, 0 , 1)$.  
By the eigen-map $\varphi=(\varphi_1, \varphi_2)$ consisting of orthonormal eigenfunctions 
\begin{align*}
& (\varphi_1(1), \varphi_1(2), \varphi_1(3), \varphi_1(4)) = \sqrt{2} (-1, 0 , 1, 0),\\ 
& (\varphi_2(1), \varphi_2(2), \varphi_2(3), \varphi_2(4)) = \sqrt{2} (0, -1, 0 , 1), 
\end{align*}
$C_4$ is realized as a square, and $(1/16\sqrt{2})\, \varphi$ satisfies \eqref{lipschitz}. 

\vskip\baselineskip
\includegraphics[bb=0 0 776 361,scale=0.4]{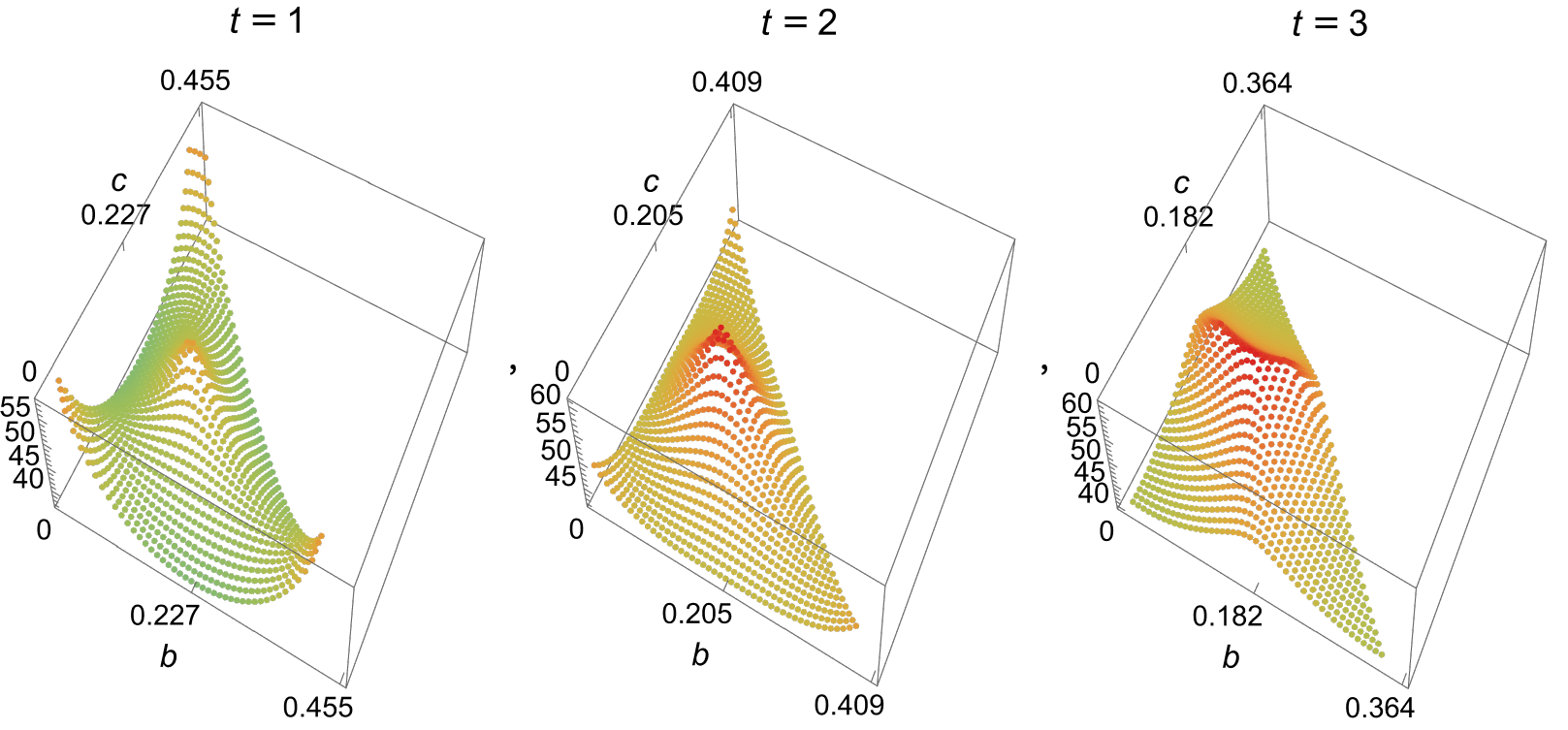}

\vskip\baselineskip
\includegraphics[bb=0 0 776 362,scale=0.4]{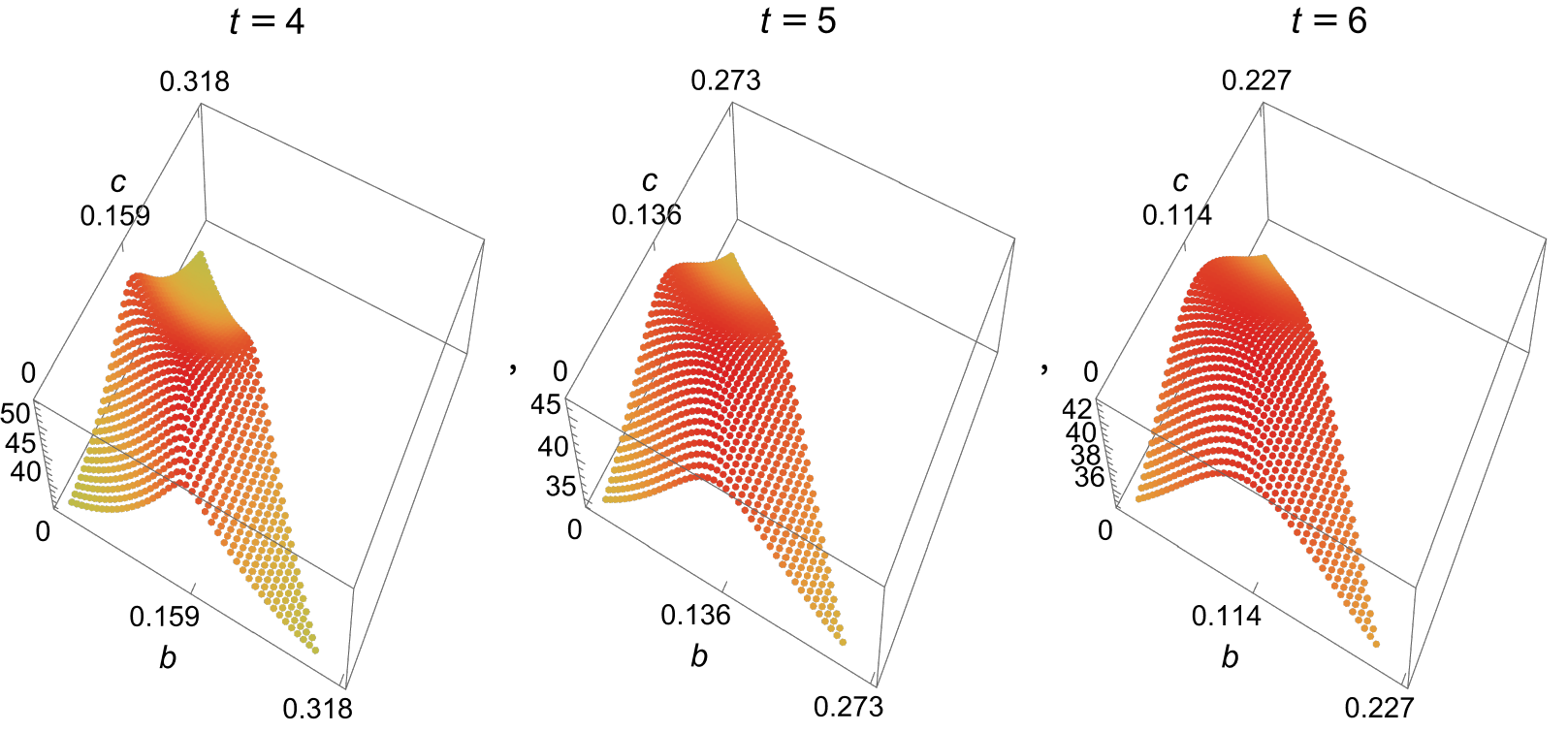}
\begin{figure}[h]
\begin{center}
\includegraphics[bb=0 0 1039 362,scale=0.4]{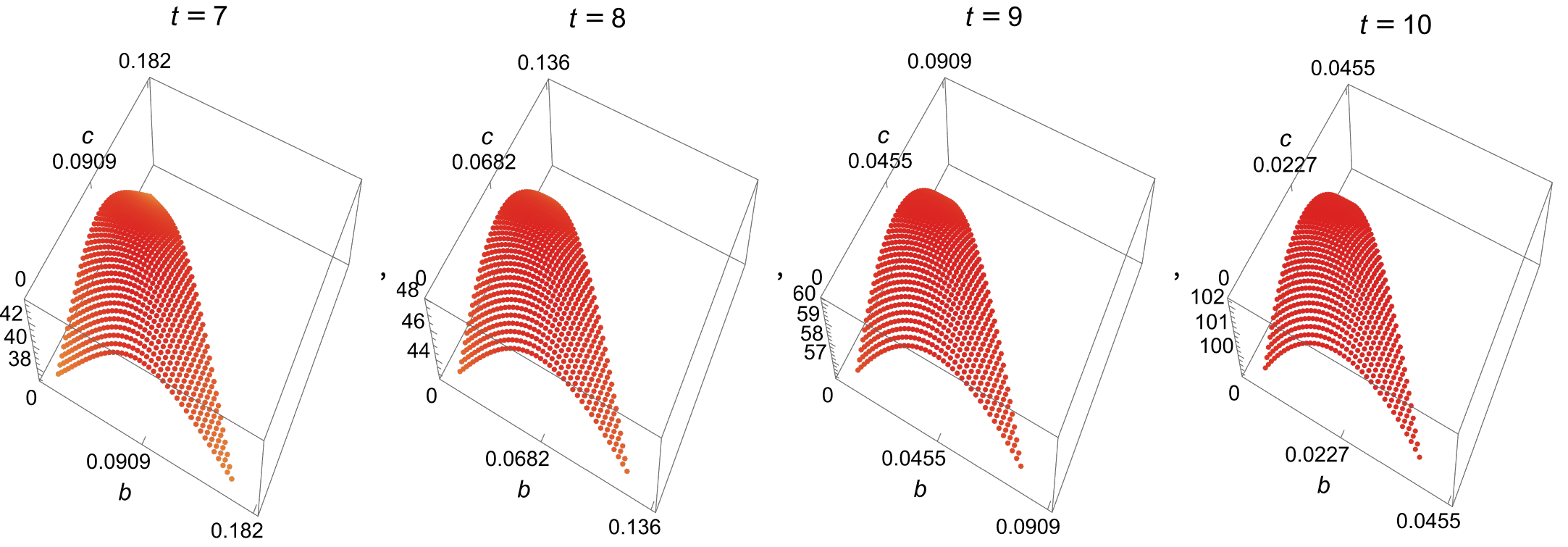}
\end{center}
\caption{Plot of the first nonzero eigenvalue on $b$-$c$ plane for each $a_t$ for $C_4$.}
\label{fig:C4at}
\end{figure}
\begin{figure}[htbp]
\begin{center}
\includegraphics[bb=0 0 392 325,scale=0.4]{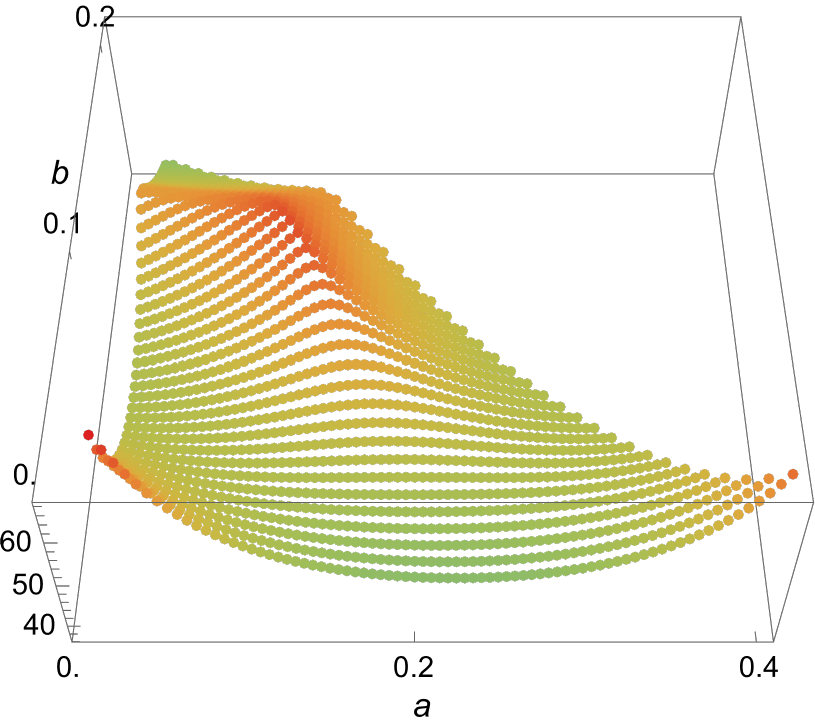}
\end{center}
\caption{Plot of the first nonzero eigenvalue on $a$-$b(=c)$ plane for $C_4$.}
\label{fig:C4b=c}
\end{figure}    
\end{Example}

\newpage
\begin{Example}
Let $G$ be a graph obtained by adding an edge $23$ to $K_{1,3}$, that is, vertices $2$ and $3$ are connected. 
Assigning length parameters $a$, $b$, $c$ and $d$ to the edges $12$, $13$, $14$ and $23$ respectively, 
the eigenvalue maximization problem is a problem of the three variables $a$, $b$ and $d$ constrained in the range $a>0$, $b>0$, $d>0$ and $a+b+d<1/2$. %($c=(1-2(a+b+d))/2$). 
On this problem, we make a numerical observation as follows.
To plot the first nonzero eigenvalue in 3D, we fix one parameter, say $d$.  
Let $d_t=(1/2)(t/n) \in (0,1/2)$ for $t=1, \cdots ,n-1$.
For each $t$, we observe the behaviour of the first nonzero eigenvalue as $a$ and $b$ vary.
Then, as Figure \ref{fig:Gdt} shows, one finds no local maximum,
and along the segment $a=b$ the first nonzero eigenvalue diverges to positive infinity as $a\, (=b)$ and $d$ approach $0$ and $1/2$, respectively ($c$ approaches $0$).
%one finds no local maximum, and along the segment $a=b$ the first nonzero eigenvalue diverges to positive infinity as $a(=b)$ and $d$ approach $0$ and $1/2$, respectively ($c$ approaches $0$).
For the values $d=1/2-10^{-6}$ and $a\simeq 2 \times 10^{-12}$, the corresponding eigen-map 
$\varphi=(\varphi_1, \varphi_2)$, consisting of orthonormal eigenfunctions 
\begin{align*}
& (\varphi_1(1), \varphi_1(2), \varphi_1(3), \varphi_1(4)) \simeq (6.5053 \times 10^{-28}, 
-1.0000, 1.0000, 1.0842 \times 10^{-16}),\\ 
& (\varphi_2(1), \varphi_2(2), \varphi_2(3), \varphi_2(4)) \simeq (6.0001 \times 10^{-9}, 
-0.0009, -0.0009, 1000.0015), 
\end{align*} 
places the vertices $1,2,3$ almost on the $x$-axis and the vertex $4$ far on the $y$-axis. 
Thus, the embedded $G$ looks like a straight $P_3$ together with a single point located far away. 

We now observe that a certain toplological degeneration may take place as the first nonzero eigenvalue diverges. 
To do so, we let $a\to 0$ and $d\to 1/2$ along the path on which the multiplicity of the first eigenvalue is two.
In Figure \ref{fig:Gb=a}, this path is the ridge-like curve. 
Then the numerical data indicate that the parameter $c$ approaches $0$ with the order $a^{1/2}$ and therefore the components of the (symmetric) matrix representing $\Delta_l$ %{(m_0,m_1)}$ 
have the following orders: 
$$
\begin{pmatrix}
+O(a^{-3/2}) & - O(a^{-5/4}) & -O(a^{-5/4}) & -O(a^{-1}) \\ * & +O(a^{-1}) & -O(a^0) & 0 \\ 
* & * & +O(a^{-1}) & 0 \\ * & 0 & 0 & +O(a^{-1}) 
\end{pmatrix}.
$$
%The component $L_{23} = d^{-1}/(\sqrt{a+d}\sqrt{b+d})$ stays bounded as $a(=b)\to 0$ and $d\to 1/2$, and the other nonzero terms diverge. 
In particular, the $(2,3)$-component is negligible compared to other components. Intuitively, this should mean that the bond strength of the edge $23$ is almost zero, that is, the graph looks like $K_{1,3}$. 
Putting aside the effect of diagonal components, the difference of diverging orders of the off-diagonal components indicates that next to $23$, the edge $14$ becomes negligible and the graph should be like the disjoint union of $P_3$ and a single isolated vertex. This interpretation of topological degeneration agrees with the geometric picture as the embedded image by the eigen-map discussed above. 

%If $a(=b)$ and $d$ are extremely close to 0 and $1/2$, respectively, then the form of the matrix representation of $\Delta_{(m_0,m_1)}$ is similar to 
%
%where all $M_{ij}$ are large positive values and $M_{12}=M_{13}$. However, $-M_{14}$ diverges to negative infinity slower than $-M_{12}=-M_{13}$, then the fourth vertex should be regarded as an isolated vertex. 
\vskip\baselineskip
\includegraphics[bb=0 0 776 360,scale=0.4]{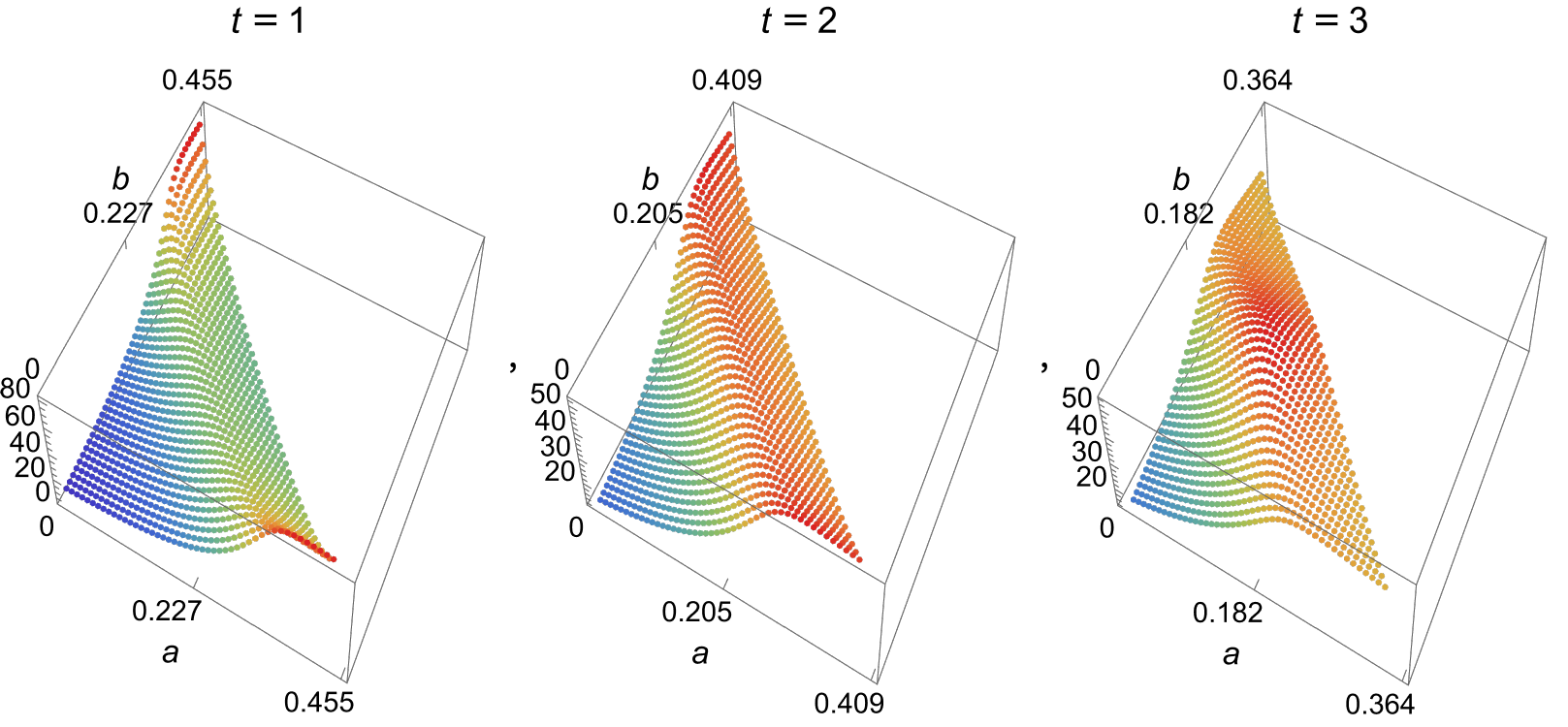}

\vskip\baselineskip
\includegraphics[bb=0 0 776 361,scale=0.4]{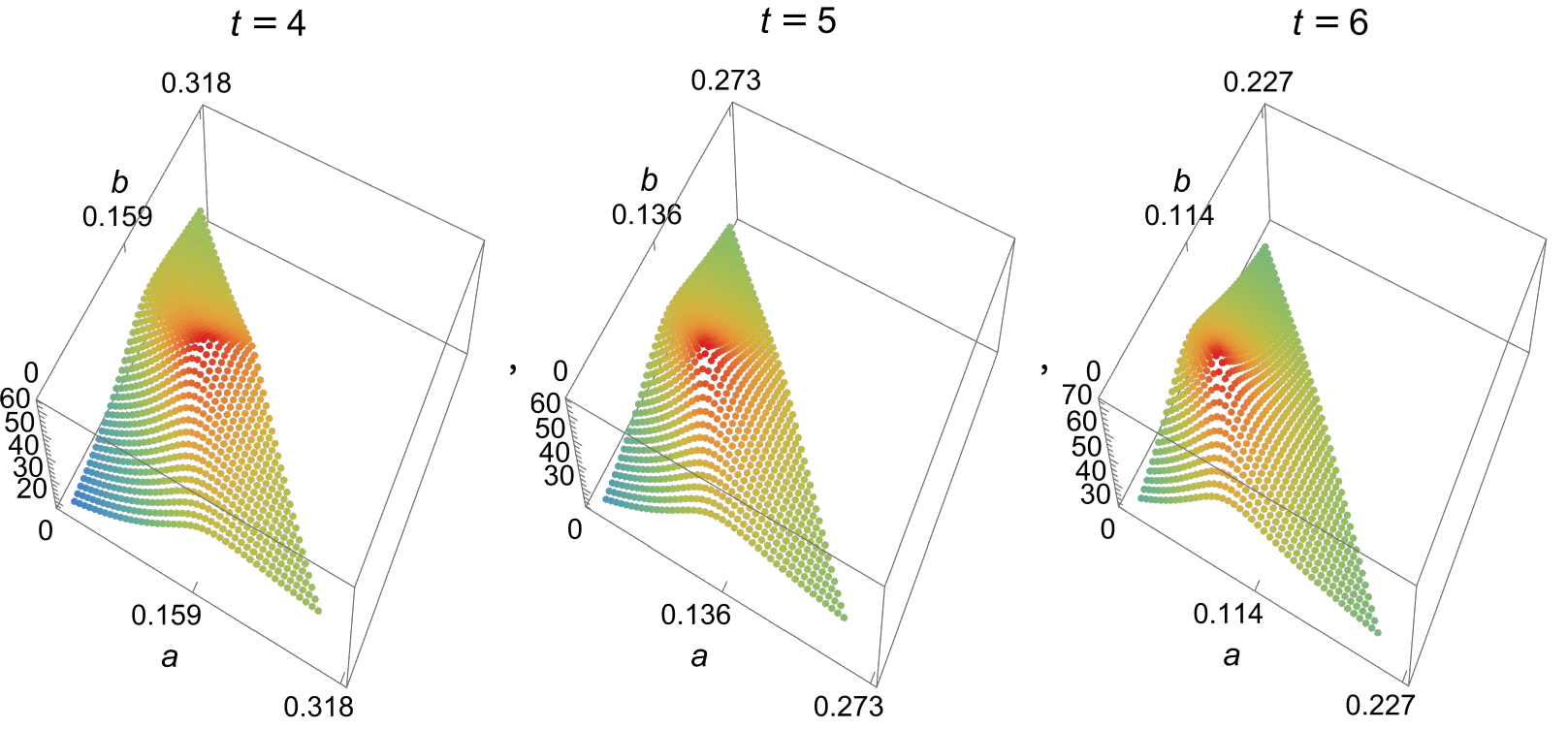}
\begin{figure}[htbp]
\begin{center}
\includegraphics[bb=0 0 1039 361,scale=0.4]{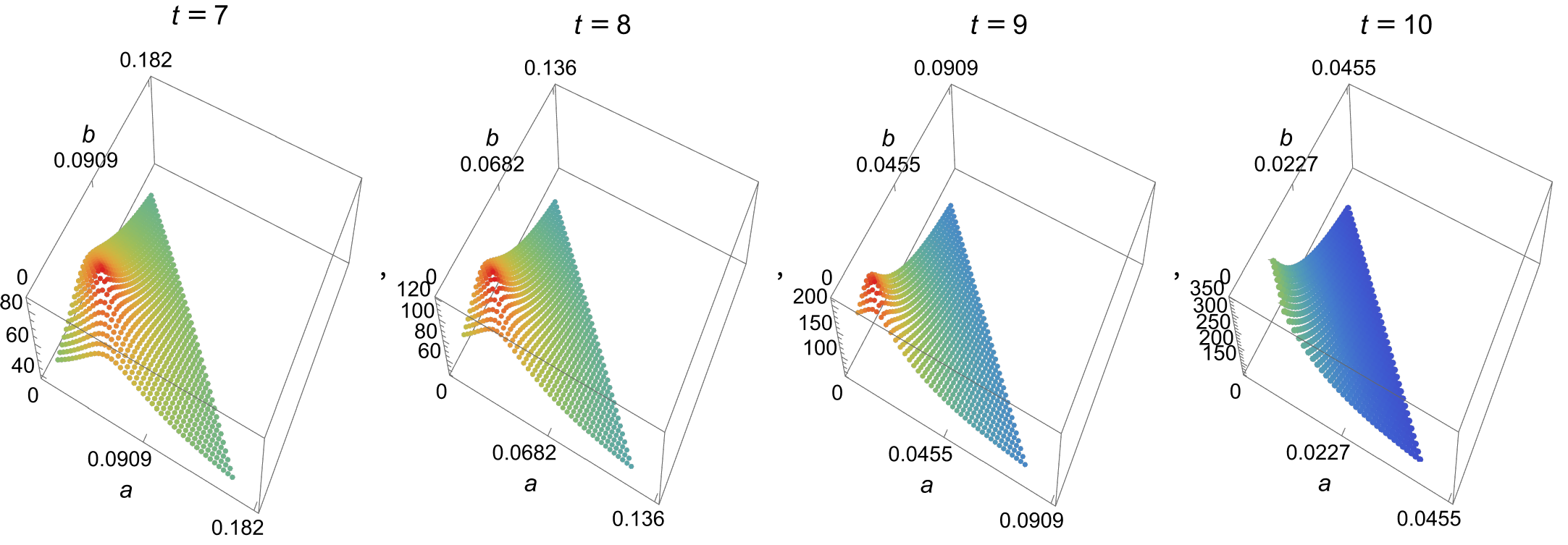}
\end{center}
\caption{Plot of the first nonzero eigenvalue on $a$-$b$ plane for each $d_t$ for $G$.}
\label{fig:Gdt}
\end{figure}
\begin{figure}[htbp]
\begin{center}
\includegraphics[bb=0 0 340 399,scale=0.4]{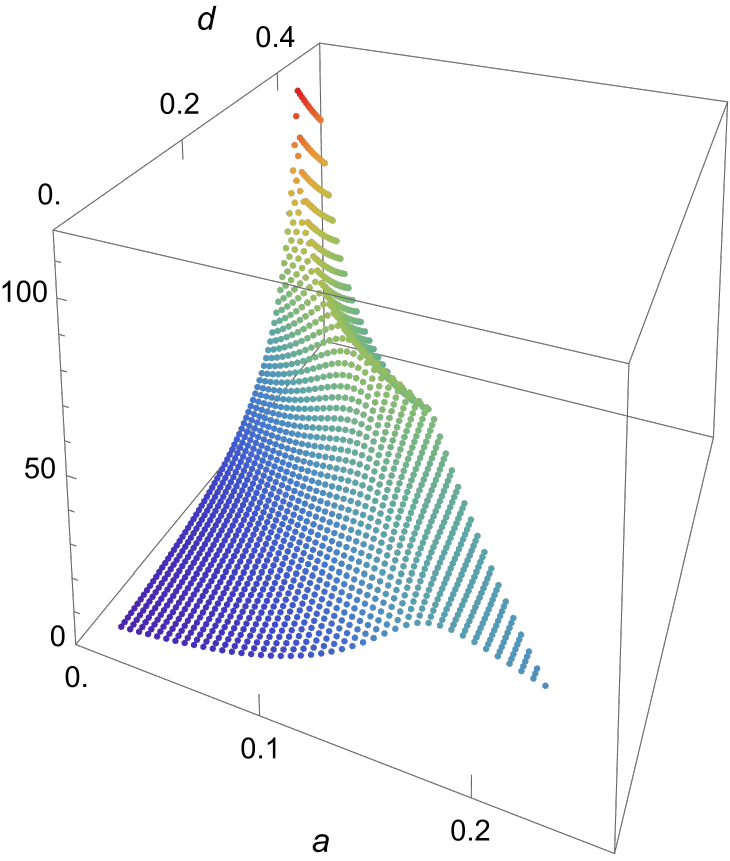} 
\end{center}
\caption{Plot of the first nonzero eigenvalue on $d$-$a(=b)$ plane for $G$.}
\label{fig:Gb=a}
\end{figure}     
\end{Example}

\newpage

\section{Case of G\"oring-Helmberg-Wappler problem}

In this section, we fix a vertex-weight $m_0 \colon V \to \mathbb{R}_{> 0}$ of unit total weight and an edge-length function 
$l \colon E \to \mathbb{R}_{> 0}$.
G\"oring-Helmberg-Wappler \cite{GoringHelmbergWappler2} introduced the following first-eigenvalue maximization problem. 

\begin{Problem}[\cite{GoringHelmbergWappler2}]\label{problemKKNspec}
Over all edge-weights $m_1$, subject to the normalization 
\begin{equation*}%\label{normalization-edge}
\sum_{uv\in E} m_1(uv)l(uv)^2 = 1,
\end{equation*}
maximize the first nonzero eigenvalue $\lambda_1(G, (m_0,m_1))$ of $\Delta_{(m_0,m_1)}$.
\end{Problem}

Note that no extremal solutions other than maximizing solutions exist, since the problem of 
minimizing the reciprocal $1/\lambda_1(G, (m_0,m_1))$ is a convex optimization problem. 
In fact, G\"oring-Helmberg-Wappler prove that a maximizing solution $m_1$ always exists, and unless $m_1$ is degenerate, 
that is, unless $m_1(uv)=0$ for some $uv\in E$, an analogue of Theorem \ref{thm-embed-var-1} holds. 

\begin{Theorem}[G\"oring-Helmberg-Wappler \cite{GoringHelmbergWappler2}]\label{embed-edgelength}
Let $m_1$ be a maximizing solution of Problem \ref{problemKKNspec}, and assume that $m_1(uv)>0$ for all $uv\in E$. 
Then there exist eigenfunctions $\varphi_1, \cdots, \varphi_N$ of the eigenvalue $\lambda_1(G, (m_0,m_1))$ 
of $\Delta_{(m_0,m_1)}$ so that the map $\varphi = (\varphi_1, \cdots, \varphi_N)\colon V\to \R^N$ satisfies 
%\begin{equation}\label{embed-edgelength-eq}
%\sum_{1 \leq i \leq d} \Phi_{\varphi_i}(uv) = l(uv)^2, \quad \forall uv \in E.
%\end{equation}
\begin{equation}\label{isometry}
l(uv) = \| \varphi(u) - \varphi(v) \|
\end{equation}
for all $uv\in E$. 
\end{Theorem}

\begin{Remark} In fact, 
G\"oring-Helmberg-Wappler \cite{GoringHelmbergWappler2} prove a more general result by allowing $m_1$ to be degenerate. 
They consider the dual problem of Problem \ref{problemKKNspec}, which reads: Over all maps 
$\varphi\colon V \to \R^{|V|}$ satisfying 
$$
\| \varphi(u) - \varphi(v) \|\leq l(uv),\,\, uv \in E\quad \mbox{and}\quad 
\overline{\varphi} = \sum _{u\in V} m_0(u) \varphi(u) = 0, 
$$
maximize the variance of $\varphi$, 
$$
\mathrm{var}(\varphi):= \sum_{u \in V} m_0(u) \| \varphi(u) \|^2.
$$
The duality implies 
\begin{equation}\label{weak duality}
\mathrm{var}(\varphi)\leq \frac{1}{\lambda_1(G, (m_0,m_1))} 
\end{equation}
for all $m_1$ and $\varphi$ satisfying the constraints. 
They prove that for fixed $m_0$ and $l$, there exist $m_1$ and $\varphi$ for which the equality sign holds 
in \eqref{weak duality}; In particular, $m_1$ and $\varphi$ are optimal solutions of Problem \ref{problemKKNspec} 
and the dual problem, respectively. 
The equality in \eqref{weak duality} implies that each component of the map $\varphi$ 
is an eigenfunction of the eigenvalue $\lambda_1(G, (m_0,m_1))$ of $\Delta_{(m_0,m_1)}$ 
and $m_1(uv)(l(uv)^2 - \| \varphi(u)-\varphi(v)\|^2)=0$ holds for all $uv\in E$. 
So, unless $m_1$ is degenerate, the last identity implies that the map $\varphi$ satisfies \eqref{isometry}. 
\end{Remark}

We give a new proof of Theorem \ref{embed-edgelength} by applying the method we used in the proof of
Theorem \ref{thm-embed-var-1}. 
This proof will be of some interest as it finds a solution of the dual problem without reference to the problem itself. 

\medskip\noindent
{\em Proof of Theorem \ref{embed-edgelength}.}\quad 
Take an analytic curve $m_1 \colon I \to (\R_{>0})^E$ such that $m_1(0)=m_1$ and 
$\sum_{uv\in E} m_1(t)(uv)\, l(uv)^2 = 1$ for all $t \in I$, where $0\in I\subset \R$ is a small interval. 
Setting $\dot{m}_0\equiv 0$ in \eqref{lambda1-derivative}, we obtain 
$$
\dot{\lambda}^{(i)}_1 = \sum_{uv \in E} \dot{m}_1(uv) \left( \varphi^{(i)}(u) - \varphi^{(i)}(v) \right)^2,\quad 
1\leq i\leq \mu, 
$$
where $\mu$ is the multiplicity of $\lambda_1(G, m_0,m_1)$. 

%For a function $\varphi \colon V \to \R$ on the vertex set, we define a function $\Phi_{\varphi} : E \to \R$ 
%on the edge set by 
%\begin{equation*}
%\Phi_{\varphi}(uv):= \left( \varphi^{(i)}(u) - \varphi^{(i)}(v) \right)^2.
%\end{equation*}

The following statement is an analogue of Claim 1 in the proof of Theorem \ref{thm-embed-var-1}:\,\, 
For any function $\rho\colon E\to \R$ satisfying $\sum_{uv\in E} \rho(uv)\, l(uv)^2=0$, 
there exists a first eigenfunction $\varphi$ of the Laplacian $\Delta_{(m_0,m_1)}$ such that
\begin{equation*}%\label{claim_GHW}
\sum_{uv\in E} \rho(uv) \left( \varphi(u) - \varphi(v) \right)^2 = 0.
\end{equation*} 
The proof is also similar. 

Let $E_1(m_0,m_1)$ denote the $\lambda_1(G,(m_0,m_1))$-eigenspace of $\Delta_{(m_0,m_1)}$, and let 
$\mathcal{C}$ be the convex hull of the set 
$$
\left\{ \left( \varphi(u) - \varphi(v) \right)^2 \mid \varphi\in E_1(m_0,m_1) \right\}
$$
in $\R^E$. 
Then by a similar argument to that for Claim 2 in the proof of Theorem \ref{thm-embed-var-1}, one can verify 
that $l^2\in \mathcal{C}$. 
That is, there exist $\varphi_k\in E_1(m_0,m_1)$, $1\leq k\leq N$, such that 
$$
l(uv)^2 = \sum_{k=1}^N \left( \varphi_k(u) - \varphi_k(v) \right)^2 = \|\varphi(u) - \varphi(v) \|^2,\quad uv\in E, 
$$
where $\varphi=(\varphi_1,\dots,\varphi_N)$. 
This concludes the proof of  Theorem \ref{embed-edgelength}. 
\hfill{$\square$}





\end{document}